\documentclass[10pt,a4paper]{amsart}

\usepackage{amsmath}
\usepackage{amssymb}

\usepackage{amscd}
\usepackage{epsfig}

\usepackage{amsthm,amssymb,amsmath,enumitem, verbatim, color, caption, faktor}

\usepackage[utf8]{inputenc}
\usepackage[T1]{fontenc}
\usepackage[all]{xy}

\usepackage{hyperref}

\newcommand{\R}{\mathbb R}

\newcommand{\C}{\mathbb C}
\newcommand{\Z}{\mathbb Z}

\newcommand{\mc}{\mathcal}

\newcommand{\End}{\operatorname{End}}
\newcommand{\id}{\operatorname{id}}

\newcommand{\grad}{\operatorname{grad}}

\newcommand{\GL}{\operatorname{GL}}

\newcommand{\SO}{\operatorname{SO}}

\newcommand{\Spin}{\operatorname{Spin}}

\newcommand{\Diff}{\operatorname{Diff}}
\newcommand{\vol}{\operatorname{vol}}
\newcommand{\tr}{\operatorname{tr}}
\newcommand{\im}{\operatorname{im}}

\usepackage{enumitem}

\newtheoremstyle{break}
  {}
  {}
  {\itshape}
  {}
  {\bfseries}
  {.}
  {\newline}
  {}

\theoremstyle{break}

\newtheorem{thm}[equation]{Theorem}
\newtheorem{prop}[equation]{Proposition}
\newtheorem{cor}[equation]{Corollary}
\newtheorem{lemma}[equation]{Lemma}
\theoremstyle{definition}

\begin{document}
\title{Stability of the spinor flow}
\author{Lothar Schiemanowski}
\address{Christian-Albrechts-Universität zu Kiel, Mathematisches Seminar, Ludewig-Meyn-Straße 4, Kiel}
\email{schiemanowski@math.uni-kiel.de}
\begin{abstract}
  We show stability of pairs of Ricci flat metrics and parallel spinor fields with respect to the spinor flow, i.e. we show that the spinor flow with initial conditions near such pairs converges to a critical point with exponential speed.
  Moreover, we show stability of certain volume constrained critical points of the spinorial energy.
\end{abstract}
\maketitle

\section{Introduction}
Given a spin manifold $M$ of dimension $n$, we consider the universal spinor bundle $\Sigma M$. This is the bundle whose sections consist of pairs of metrics $g \in \Gamma(\odot^2_+ T^*M)$ and spinor fields $\varphi \in \Gamma(\Sigma_g M)$.
We denote by $\mc{N}$ the set
$$\{ (g, \varphi) \in \Gamma( \Sigma M) : |\varphi| = 1\}$$
and define the spinorial energy functional
$$\mc{E}: \mc{N} \to \R$$
$$\mc{E}(g, \varphi) = \frac{1}{2}\int_M |\nabla^g \varphi|^2 \vol_g.$$
If the dimension of $M$ is at least three, the only critical points of $\mc{E}$ are absolute minimizers. This implies $\nabla^g \varphi = 0$ for a critical point $(g, \varphi)$, i.e. $\varphi$ is a parallel spinor with respect to the metric $g$. Existence of parallel spinors is a strong constraint on the metric $g$. Indeed, such a metric is necessarily Ricci flat and of special holonomy. Conversely, Ricci flat manifolds with special holonomy admit a parallel spinor.  Given that manifolds with such metrics are difficult to construct, it is natural to consider the negative gradient flow
$$\partial_t \Phi_t = Q(\Phi_t)$$
of $\mc{E}$ to find such a metric. Here $Q: \mc{N} \to T \mc{N}$ is the negative gradient of $\mc{E}$ with respect to the natural $L^2$ metric on $\mc{N}$. It turns out that $Q$ is weakly elliptic and has negative symbol.
The spinorial energy and the associated negative gradient flow, called spinor flow, were first examined in \cite{Ammann2015}. There, short time existence of this flow on closed manifolds was established.
From here on we assume $M$ to be a closed manifold and $\dim M = n \geq 3$.
We will prove that critical points of $\mc{E}$, i.e. pairs of Ricci flat metrics and parallel spinor fields, are stable with respect to the spinor flow, that is:
\begin{thm}
Suppose $\bar{\Phi} = (\bar{g}, \bar{\varphi})$ is a critical point of $\mc{E}$ and suppose $\bar{g}$ has no Killing fields. Then there exists a $C^{\infty}$ neighborhood $U$ of $\bar{\Phi}$, such that a solution of the negative gradient flow $\Phi_t$ with initial condition $\Phi_0 = \Phi$ smoothly converges to a critical point. In any $C^k$ norm the speed of convergence is exponential.
\end{thm}
A Ricci flat manifold is up to a finite covering a product of irreducible Ricci-flat manifolds and a flat torus. Hence the condition on the Killing fields can also be read as saying that $\bar{g}$ has no torus factor. This can for instance be ruled out by the topological condition that the fundamental group of $M$ be finite.
The strategy of the proof will be roughly as follows: first, we establish a Łojasiewicz-Simon type inequality for the spinorial energy. This inequality implies exponential decay of the energy along the flow. We will then show that this implies convergence to a critical point. The inequality depends in its optimal form on the fact that the critical set of $\mc{E}$ is smooth. This was shown in \cite{Ammann2015b}.

We will also consider the stability of volume constrained critical points. A section $\Phi \in \mc{N}$ is a {\em volume constrained critical point}, if
$$\frac{d}{dt}\Bigr|_{t=0} \mc{E}(\Phi_t) = 0$$
for all volume preserving variations $\Phi_t$ of $\Phi$. Such a critical point evolves under the spinor flow by rescaling. A {\em volume constrained minimizer} $\Phi = (g, \varphi)$ is a volume constrained critical point, such that for any $\Psi \in \mc{N}$ close to $\Phi$ we have $\mc{E}(\Phi) \leq \mc{E}(\Psi)$, provided the metrics induced by $\Phi$ and $\Psi$ have equal volume.
\begin{thm}
Let $\bar{\Phi} = (\bar{g}, \bar{\varphi})$ be a volume constrained minimizer of $\mc{E}$. Suppose that the critical set near $\bar{\Phi}$ is a manifold and suppose $\bar{g}$ has no Killing fields. Then there exists a $C^{\infty}$ neighborhood $U$ of $\bar{\Phi}$, such that the volume normalized spinor flow converges smoothly to a volume constrained minimizer, if the initial condition is in $U$. The convergence speed in $C^k$ is exponential. 
\end{thm}
The strategy for the proof is essentially the same as in the case of critical points. However, here both the condition on Killing fields and the assumption that the critical set near $\bar{\Phi}$ is a manifold are strong restrictions. Indeed, suppose $(g, \varphi)$ is such that
$$\nabla^g_X \varphi = \lambda X \cdot \varphi \text{ for all } X \in \Gamma(TM)$$
with $\lambda \in \R$.
Then $(g,\varphi)$ is a volume constrained critical point. The spinor $\varphi$ is called a Killing spinor.
If $g$ carries a Killing spinor, then the cone $((0,\infty)\times M, dr^2 + r^2g)$ carries a parallel spinor.
A large class of metrics with Killing spinors is supplied by Sasaki--Einstein manifolds.
Since the Reeb vector field is a Killing vector field, all Sasaki--Einstein manifolds carry Killing fields.
Furthermore, the moduli space of Sasaki--Einstein manifolds is not known to be smooth in general.
What's more, in contrast to the space of parallel spinors, whose dimension is locally constant under Ricci flat deformations of the metric, the dimension of the space of killing spinors can jump under Einstein deformations of the metric. Indeed, a 3-Sasakian manifold admits three linearly independent Killing spinors. Van Coevering found that a toric 3-Sasakian manifold has Einstein deformations $g_t$, such that the space of Killing spinors is two-dimensional for any $t\neq 0$.

Since the spinor flow is a generalization of the heat flow for $G_2$-structures introduced in \cite{Weiss2012}, our result is a generalization of the stability result proven there. However, the arguments of our proof are closer in spirit to the proofs in \cite{Haslhofer2011}, \cite{Haslhofer2014}, \cite{Krncke2014}, where stability of Ricci-flat and Einstein metrics with respect to the Ricci flow is shown.

\section*{Acknowledgements}
The author thanks Hartmut Weiß for posing the problem and numerous discussions related to it.

\section{The universal spinor bundle and the spinor flow}
For convenience and completeness, we recall the precise definitions of the spinor flow as well as results on short time existence of the flow. Details may be found in \cite{Ammann2015}. 
We defined the spinor energy to be a functional on sections of the universal spinor bundle. We will now construct this universal spinor bundle.
Before we do this, let us first recall the ordinary spinor bundle on a spin manifold. 
The orientation preserving component of the general linear group $\GL_+(n)$ has fundamental group $\Z_2$ and hence there exists a universal double covering group $\widetilde{\GL}_+(n)$ together with a covering map $\xi: \widetilde{\GL}_+(n) \to \GL_+(n)$, which is also a homomorphism.
Let $M$ be a spin manifold of dimension $n$. By this we mean a manifold $M$ and a $\widetilde{\GL}_+(n)$ principal bundle $\tilde{P}$ which covers the $\GL_+(n)$ frame bundle $P$, $\pi: \tilde{P} \to P$, so that for any $g \in \widetilde{\GL}_+(n), p \in \tilde{P}$ we have
$$\pi(p \cdot g) = \pi(p) \cdot \xi(g).$$
Now let $g$ be a metric on $M$. The metric induces a reduction of the structure group of $P$ to the oriented orthonormal frame bundle $P_{\SO(n)}$. The group $\Spin(n) = \xi^{-1}(\SO(n))$ is called the spin group. Thus the structure group of $\tilde{P}$ reduces to $\Spin(n)$ and we call this bundle $P_{\Spin(n)}$, which double covers $P_{\SO(n)}$.
Now we define the (complex) spinor bundle as the associated vector bundle
$$\Sigma_g M = P_{\Spin(n)} \times_{\Delta_n} \Sigma_n$$
where $\Delta_n : \Spin(n) \to \End(\Sigma_n)$, $\Sigma_n = \C^{2^{[n/2]}}$, is the standard complex spin representation. Up to scaling, there exists one $\Spin(n)$ invariant Hermitian product on $\Sigma_n$. This turns $\Sigma_g M$ into a Hermitian bundle.
The universal spinor bundle gives us a way to compare spinors over different metrics.
Recalling that
$$\faktor{\GL_+(n)}{\SO(n)} \cong \odot^2_+ \R^n,$$
we conclude
$$\odot^2_+ T^*M = P \times_{\GL_+(n)} \faktor{\GL_+(n)}{\SO(n)} = \tilde{P} \times_{\widetilde{\GL}_+(n)} \faktor{\widetilde{\GL}_+(n)}{\Spin(n)} = \faktor{\tilde{P}}{\Spin(n)}.$$
We define
$$\Sigma M = \tilde{P} \times_{\Delta_n} \Sigma_n.$$
This is a vector bundle over $\odot^2_+ T^*M$, i.e. we have the structure of two nested fibrations:
$$\Sigma M \xrightarrow{\pi_{\Sigma}} \odot^2_+ T^*M \xrightarrow{\pi_{\mc{M}}} M.$$
Given a metric $g$ we can identify $\pi_{\Sigma}^{-1}(g)$ and $\Sigma_g M$. Using this identification any element $\Phi \in\Sigma M$ can be considered as a pair of a metric $g_{\Phi} = \pi_{\Sigma}(\Phi)$ and a spinor $\varphi_{\Phi} \in \Sigma_{g_{\Phi}} M$.
As above we also get a Hermitian inner product $h$ on $\Sigma M$.
We denote by $\langle \cdot, \cdot \rangle = \operatorname{Re} h$ the real part of $h$ and $| \cdot |$ the associated norm.
Now the definition
$$\mc{N} = \{\Phi \in \Gamma(\Sigma M) : |\Phi| = 1\}$$
from the introduction is fully explained. To make sense of the gradient of $\mc{E}$ we need to compute the tangent spaces of $\mc{N}$. For this we need to compare spinors in different fibers $\Sigma_{g_1} M$ and $\Sigma_{g_2} M$. This can be done using the Bourgignon--Gauduchon connection.

Suppose we have a vector space $V$ and two inner products $\langle \cdot, \cdot \rangle_1$ and $\langle \cdot, \cdot \rangle_2$. Then there exists a unique endomorphism $A^2_1 : V \to V$, such that
$$\langle v, w \rangle_2 = \langle A_1^2 v, w \rangle_1 \text{ for all } v,w \in V.$$
Denote by $B_1^2$ the square root of $A_1^2$. The operator $B_1^2$ maps orthonormal bases of $(V, \langle \cdot, \cdot \rangle_1)$ to orthonormal bases of $(V, \langle \cdot, \cdot \rangle_2)$. Since no choices are involved and $B_1^2$ depends smoothly on the inner products, we can transfer this construction to Riemannian manifolds $(M,g_i)$, $i=1,2$, and consequently obtain a smooth principal bundle isomorphism
$$P^{g_2}_{\SO(n)} \to P^{g_1}_{\SO(n)}.$$
This map lifts to the spinor bundle and hence induces a isomorphism $\hat{B}^{g_2}_{g_1}:\Sigma_{g_2} M \to \Sigma_{g_1} M$.
Since the metric on $\Sigma_n$ is $\Spin(n)$-invariant, this is an isometry.
Notice that the restriction of $\hat{B}_{g_1}^{g_2}$ to a fibre over a point $x \in M$ only depends on the scalar products $g_1(x), g_2(x)$ on $T_xM$.
We now have a canonical isometry between two spinor bundles over the same manifold with two distinct metrics. From this we can derive a horizontal distribution
$$\mc{H}_{\Phi} = \left\{ \frac{d}{dt}\Bigr|_{t=0} \hat{B}^{g}_{g_t} \varphi \; \Big| \; g: (-\epsilon, \epsilon) \to \odot^2_+ T_x^*M, g(0) = g \right\} \subset T_{\Phi} \Sigma M $$
where $\Phi = (g, \varphi) \in \Sigma M_x$. By construction, $\mc{H}_{\Phi} \cong \odot^2 T_x^*M$.
This distribution yields a splitting of the tangent bundle
$$T_{\Phi} \Sigma M_x = \mc{H}_{\Phi} \oplus \Sigma_g M  \cong \odot^2 T_x^*M \oplus \Sigma_g M.$$
Turning to sections of the universal spinor bundle, this implies that for $\Phi = (g, \varphi) \in \Gamma (\Sigma M)$ we have the splitting
$$T_{\Phi} \Gamma(\Sigma M) = \Gamma(\odot^2 T^* M) \oplus \Gamma( \Sigma_g M)$$
and if $\Phi \in \mc{N}$
$$T_{\Phi} \mc{N} = \Gamma(\odot^2 T^* M) \oplus \Phi^{\perp},$$
where
$$\Phi^{\perp} = \{ \psi \in \Gamma(\Sigma_g M) : \langle \varphi, \psi\rangle \equiv 0 \}.$$
Now we define for $\Phi \in \Gamma (\Sigma M)$ and $\Psi_1, \Psi_2 \in T_{\Phi} \Gamma( \Sigma M)$
$$\left( \Psi_1, \Psi_2 \right)_{L^2} = \int_M g(h_1, h_2) \vol_g + \int_M \langle \psi_1, \psi_2 \rangle_{\Sigma_g M} \vol_g$$
where $(h_i, \psi_i) \in \Gamma(\odot^2_+ T^*M) \oplus \Gamma (\Sigma_g M)$ are the sections corresponding to $\Psi_i$ according to the isomorphisms above. From now on we will use these isomorphisms implicitly.
Now the negative gradient
$$Q: \mc{N} \to T\mc{N}$$
is defined by the property
$$\left(Q(\Phi), \Psi \right) = - \frac{d}{dt}\Bigr|_{t=0} \mc{E}(B^g_{g+th} (\varphi + t \psi)),$$
where $\Phi = (g, \varphi) \in \mc{N}$ and $\Psi = (h, \psi) \in T_{\Phi} \mc{N}$.

\section{Diffeomorphism invariance, the gauged spinor flow and a slice theorem}
We denote by $\Diff_s(M)$ the group of spin diffeomorphisms, i.e. the orientation preserving diffeomorphisms of $M$, which lift to $\tilde{P}$. To be more precise, by a lift of a orientation preserving diffeomorphism $f: M \to M$ to $\tilde{P}$, we mean a lift of the map
$$ P_x \ni [e_1, ..., e_n] \mapsto [Df e_1,..., Df e_n] \in P_{f(x)}$$
induced by $f$ on the oriented frame bundle $P$ to the topological spin bundle $\tilde{P}$.
Since $\tilde{P}$ is a $\Z_2$ bundle over $P$, there is a choice of lift and the group of lifts of spin diffeomorphisms $\widehat{\Diff}_s(M)$ fits into an exact sequence
$$0 \to \Z_2 \to \widehat{\Diff}_s(M) \to \Diff_s(M) \to 0.$$
The group $\widehat{\Diff}_s(M)$ acts on $\Gamma( \Sigma M)$ in the following way. Let $F \in \widehat{\Diff}_s(M)$, $\Phi = (g, \varphi) \in \Gamma(\Sigma M)$. The map $F : \tilde{P} \to \tilde{P}$ is a lift of a diffeomorphism $f: M \to M$. Restricting $\tilde{P}$ to $P_{\Spin(n)}^g$ we obtain an isomorphism
$$F : P_{\Spin(n)}^g \to P^{f_* g}_{\Spin(n)}.$$
Then we define locally
$$F_*\varphi = [F \circ b \circ f^{-1}, \varphi \circ f^{-1}] \in \Gamma(\Sigma_{f_* g} M),$$
if $\varphi = [b, \varphi]$, $b$ a local section of $P_{\Spin(n)}^g$, $\varphi$ a $\Sigma_n$ field. The push forward preserves the metric in the following sense:
$$|F_* \varphi|_{f_* g} (x)  = |\varphi|_g (f^{-1}(x)).$$
In particular $F_*$ preserves $\mc{N}$. Moreover, we have
$$\mc{E}(F_* \Phi) = \mc{E}(\Phi),$$
$$Q(F_* \Phi) = F_* Q(\Phi).$$
In particular, the spinor flow is not strongly parabolic, since $Q$ is invariant under an infinite dimensional group. This invariance is reflected on the infinitesimal level by the following Bianchi-type identity
$$\lambda_{g,\varphi} Q(g,\varphi) = 0$$
where
$$\lambda_{g, \varphi} : \Gamma(\odot^2 T^*M) \oplus \Gamma(\Sigma_g M) \to \Gamma(TM)$$
is defined as the formal adjoint of
$$\lambda_{g,\varphi}^* : \Gamma(TM) \to \Gamma(\odot^2 T^*M) \oplus \Gamma(\Sigma_g M)$$
$$X \mapsto \left(2 \delta^*_g X^{\flat}, \nabla^g_X \varphi - \frac{1}{4} dX^{\flat} \cdot \varphi\right) = \left( \mc{L}_X g, \tilde{\mc{L}}_X \varphi\right) =: \tilde{\mc{L}}_X \Phi.$$
Indeed, the tangent space of the orbit $\widehat{\Diff}_s(M).(g,\varphi)$ is the image of $\lambda_{g,\varphi}^*$.
At a critical point $(g,\varphi) \in \Gamma( \Sigma M)$, we get the following exact sequence
$$0 \to \Gamma(TM) \xrightarrow{\lambda_{g, \varphi}^*} \Gamma(\odot^2 T^* M) \oplus \Gamma(\varphi^{\perp}) \xrightarrow{L_{g,\varphi}} \Gamma(\odot^2 T^* M) \oplus \Gamma(\varphi^{\perp}) \xrightarrow{\lambda_{g, \varphi}} \Gamma(TM) \to  0,$$
where $L_{g,\varphi} = DQ (g,\varphi)$. It turns out that with
$$X_{\bar{g}} : \Gamma(\odot^2 T^*M) \to \Gamma(TM)$$
$$g \mapsto -2 (\delta_{\bar{g}} g)^{\sharp}$$
and $\bar{g}$ any given metric the operator
$$\tilde{Q}_{\bar{g}}(\Phi) = Q(\Phi) + \lambda^*_{g,\varphi}(X_{\bar{g}}(\Phi))$$
is strongly parabolic for any $\Phi = (\bar{g}, \varphi) \in \mc{N}_{\bar{g}}$ and hence the flow
$$\partial_t \Phi_t = \tilde{Q}(\Phi_t)$$
exists for short time. We call this flow {\em gauged spinor flow} or {\em spinor-DeTurck flow}.
Moreover, the spinor flow and the gauged spinor flow differ only by a family of diffeomorphisms, i.e.
if $\Phi_t = (g_t, \varphi_t)$ is a solution of the spinor flow and $\tilde{\Phi}_t = (\tilde{g}_t, \tilde{\varphi}_t)$ is a solution of the gauged spinor flow with $\Phi_0 = \tilde{\Phi}_0$, then there exists a family $F_t \in \widehat{\Diff}_s(M)$, induced by $f_t \in \Diff(M)$, such that
$$\tilde{\Phi}_t = F_{t*} \Phi_t.$$
This family obeys the partial differential equation
$$\partial_t f_t = P_{g_t, \bar{g}}(f_t)$$
with initial condition $f_0 = \id_M$, where
$$P_{g,\bar{g}} : \mc{C}^{\infty}(M,M) \to T \mc{C}^{\infty}(M,M)$$
$$f \mapsto -df (X_{f^* \bar{g}} (g)).$$
For future reference we note that the linearization of $P_{\bar{g},\bar{g}}$ at $\id_M$ is given by
$$\Gamma(TM) \ni X \mapsto -4 (\delta_{\bar{g}} \delta^*_{\bar{g}} X^{\flat})^{\sharp} \in \Gamma(TM).$$
Because
$$T_{\Phi} \Gamma(\Sigma M) = \ker \lambda_{\Phi} \oplus \im \lambda_{\Phi}^* = \ker \lambda_{\Phi} \oplus T_{\Phi} \widehat{\Diff}_s(M).\Phi,$$
we can consider $\ker \lambda_{\Phi}$ to be an infinitesimal slice to the diffeomorphism action.
 Indeed, we will prove that, in a weak sense, $\ker \lambda_{g,\varphi}$ parametrizes a slice in a simple way. To see this we first need a parametrization of $\Gamma(\Sigma M)$ by the set $T_{(g,\varphi)} \Gamma(\Sigma M) = \Gamma(\odot^2 T^*M) \oplus \Gamma(\Sigma_g M)$ near $(g,\varphi)$. This will be frequently useful and throughout the rest of the article $\Xi = \Xi_{g,\varphi}$ denotes this parametrization.
We define
$$\Xi_{g,\varphi}: (U_g \subset \Gamma(\odot^2 T^*M)) \times \Gamma(\Sigma_g M) \to \Gamma(\Sigma M)$$
$$(h, \psi) \mapsto (g+h, \hat{B}^g_{g+h}(\varphi + \psi))$$
and its inverse
$$\Xi^{-1}: \Gamma(\Sigma M) \to U_g \times \Gamma(\Sigma_g M)$$
$$(g', \varphi') \mapsto (g'-g, \hat{B}^{g'}_g (\varphi') - \varphi).$$
Here $U_g = \{h \in \Gamma(\odot^2 T^*M) : g+h \text{ is a metric}\}$.
In terms of this parametrization we can formulate the following slice theorem:
\begin{prop}
Let $\Phi = (g,\varphi) \in \Gamma(\Sigma M)$ and assume $g$ has no Killing fields. Then there exists a $C^{k+1,\alpha}$ neighborhood $U$ of $\Phi$, such that for any $\tilde{\Phi} \in U$, there exists a $C^{k+2,\alpha}$ diffeomorphism $f: M\to M$, such that
$$\lambda_{\Phi}(\Xi^{-1}(F^* \tilde{\Phi})) = 0.$$
\end{prop}
\begin{proof}
We base the proof on \cite{Viaclovsky13}, theorem 3.6.
Consider the map
$$G: \Gamma^{k+1, \alpha}(\odot^2 T^*M \oplus \Sigma_g M) \times \Gamma^{k+2, \alpha}(TM) \to \Gamma^{k,\alpha}(TM)$$
$$((h, \psi), X) \mapsto \lambda_{\Phi}(\phi^{X*}_1(g+h, B^g_{g+h}(\varphi + \psi)))$$
Then the derivative of $G$ at $((0,0), 0)$ in $X$ direction is given by
$$\frac{d}{dt}\Bigr|_{t=0}  \lambda_{\Phi}(\phi^{tV*}_1 \Phi) = \lambda_{\Phi}(\tilde{\mc{L}}_V \Phi) = \lambda_{\Phi} \lambda_{\Phi}^* V$$
for $V \in \Gamma^{k+2, \alpha}(TM)$. Since $g$ posesses no Killing fields, $\lambda_{\Phi} \lambda_{\Phi}^*$ is injective, because in the first component $\lambda_{\Phi} \lambda_{\Phi}^* X$ is just $\delta_g \delta_g^* X^{\flat}$. Additionally, $\lambda_{\Phi} \lambda_{\Phi}^*$ is an elliptic operator. It is selfadjoint and hence it must also be surjective. Thus we may apply the implicit function theorem and we find that there exists a neighborhood $U \subset \Gamma^{k+1, \alpha}(\odot^2 T^*M \oplus \Sigma_g M)$ of $(0,0)$ and a map $H: U \to \Gamma^{k+2, \alpha}(TM)$, such that $G((h,\psi), H(h,\psi)) = 0$.
Now let $\tilde{\Phi} \in \Xi(U)$.
Then denote by $f$ the time-$1$ map of the vector field $H(\Xi^{-1}(\tilde{\Phi}))$. Then
$$\lambda_{\Phi}(F^*(\Xi^{-1}(\tilde{\Phi}))) = 0$$
by construction.
The statement then follows, because
$$F^*(\Xi^{-1}(\tilde{\Phi})) = \Xi^{-1}(F^*(\tilde{\Phi})).$$
\end{proof}

\section{Volume normalized spinor flow}
Volume constrained critical points evolve by rescaling under the spinor flow. We expect similar behavior near such a point. To address convergence questions in this situation, it is thus useful to rescale the solutions to a fixed volume. In this section, we introduce the volume normalized spinor flow and describe its evolution equation.
Let $\Phi_t = (g_t, \phi_t)$ be a solution to the spinor flow. We denote by $\mu(t)$ the normalizing factor $\left(\int_M \vol_{g_t}\right)^{-2/n}$. Then $\int_M \vol_{\mu(t) g_t} = 1$.
Now let $\tilde{\Phi}(t) = (\tilde{g}(t), \tilde{\varphi}(t))$, where
$$\tilde{g}(t) = \mu (\tau(t)) g_{\tau(t)},$$
$$\tilde{\varphi}(t) =  \hat{B}^{g_{\tau(t)}}_{\mu(\tau(t)) g_{\tau(t)}} (\varphi_{\tau(t)}),$$
where $\tau: I \subset \R \to J \subset \R$ is some time reparametrization.
Then we have
$$\partial_t \tilde{g}_t = \dot{\mu}(\tau(t)) \tau'(t) g_{\tau(t)} + \mu(\tau(t)) \dot{g}_{\tau(t)} \tau'(t),$$
$$\partial_t \tilde{\varphi}_t = \hat{B}^{g_{\tau(t)}}_{\mu(\tau(t)) g_{\tau(t)}} (\dot{\varphi}_{\tau(t)}) \tau'(t).$$
Solving a separable ordinary differential equation, we can arrange $\tau'(t) \mu(\tau(t)) = 1$.
We call $\tilde{\Phi}_t$ with this choice of time rescaling the {\em volume normalized spinor flow}.
For any $h\in \Gamma(\odot^2 T^*M)$, we denote by $\mathring{h}$ the tensor 
$$h - \frac{\int_M \tr_g h \vol_g}{n \int_M \vol_g} g.$$
Since $\tilde{g}_t$ has constant volume $1$, it follows that $\int_M \partial_t g_t \vol_{g_t} = 0$. 
Thus we have
$$\partial_t \tilde{g}_t = \mathring{Q}_1(g_{\tau(t)}, \varphi_{\tau(t)}).$$
By corollary 4.5 in \cite{Ammann2015}, we moreover have $Q_1(c^2 g, \hat{B}^g_{c^2 g}(\varphi)) = Q_1(g, \varphi)$, which implies
$$\partial_t \tilde{g}_t = \mathring{Q}_1(g_{\tau(t)}, \varphi_{\tau(t)}) = \mathring{Q}_1\left(\mu (\tau(t)) g_{\tau(t)}, \hat{B}^{g_{\tau(t)}}_{\mu(\tau(t)) g_{\tau(t)}} (\varphi_{\tau(t)})\right) = \mathring{Q}_1(\tilde{\Phi}_t).$$
Again by corollary 4.5 in op. cit., we have $Q_2(c^2 g, \hat{B}^g_{c^2 g} (\varphi)) = c^{-2} \hat{B}^g_{c^2 g} Q_2(g, \varphi)$. Thus
$$\partial_t \tilde{\varphi}_t = \mu(t)^{-1} \hat{B}^{g_{\tau(t)}}_{\mu(\tau(t)) g_{\tau(t)}} (Q_2(g_{\tau(t)}, \varphi_{\tau(t)})) = Q_2(\tilde{\Phi}_t).$$
We define
$$\mathring{Q}(\Phi) = (\mathring{Q}_1(\Phi), Q_2(\Phi))$$
and can rewrite the evolution of $\tilde{\Phi}_t$ as
$$\partial_t \tilde{\Phi}_t = \mathring{Q}(\tilde{\Phi}_t).$$
Since $\mathring{Q}$ is the negative gradient of $\mc{E}$ restricted to the set
$$\mc{N}^1 = \left\{ \Phi = (g, \varphi) \in \mc{N} : \int_M \vol_g = 1\right\},$$
we conclude that the volume normalized spinor flow coincides with the negative gradient flow of $\mc{E}$ restricted to $\mc{N}^1$.

\section{Analytical setup}
In the following proof of stability we will analyze three flows: the spinor flow, the gauged spinor flow and the mapping flow. Each of these flows is defined on an infinite dimensional manifold rather than a vector space and we feel it is appropiate to clarify our analytic setup, so that we can proceed in a somewhat more formal manner later on without bypassing rigor altogeher. 

The set of unit spinors $\mc{N}$ forms a Fréchet manifold with the $\mc{C}^{\infty}$ topology. We will however never use this topology directly. Instead, we will typically restrict to a chart and work with the Sobolev or $C^{k,\alpha}$ topologies. We do this as follows.
Fix $\Phi_0 = (g_0, \varphi_0) \in \Gamma(\Sigma M)$. We already constructed the chart 
$$\Xi^{-1}_{\Phi_0} : U \subset \Gamma(\Sigma M) \to V \subset \Gamma(\odot^2 T^*M) \oplus \Gamma(\Sigma_{g_0} M).$$
The metric $g_0$ then induces the usual $H^k$ and $C^{k,\alpha}$ norms on $\Gamma(\odot^2 T^*M) \oplus \Gamma(\Sigma_g M)$ and we simply pull them back via the chart.
Locally we can now consider the spinor energy $\mc{E}$ as a map $V \to \R$ and $Q$ as a map $V \to V$.
Whenever we use a $C^k$ or $H^s$ norm we implicitly use this construction. In particular, when we write $\|\Phi - \Phi_0\|_{X}$ for a fixed $\Phi_0$ and a nearby $\Phi$, we mean  $\|\Xi^{-1}_{\Phi_0} (\Phi)\|_{X}$, where $X$ is one of the discussed Banach spaces.

For the mapping flow we proceed in a similar manner. Note first that for $f_0 \in \mc{C}^{\infty}(M,M)$, there is a local chart around $f_0$ given by
$$U \subset \mc{C}^{\infty}(M,M) \to V \subset \Gamma(f_0^* TM)$$
$$f \mapsto (x \mapsto (\exp_{f_0(x)})^{-1}(f(x))),$$
where $\exp$ is the exponential map of some Riemannian metric on $g$ and $V$ is a neighborhood of the $0$ section in $TM$, such that $exp_x$ is a diffeomorphism from $V_x = T_x M \cap V$ to $\exp (V_x)$ for every $x \in M$. 
Then we define
$$U = \{f: M\to M \Big| (f_0,f)(M) \subset \exp(V)\}.$$
We can define appropiate norms in the standard manner using some Riemannian metric on $M$, for example
$$\left(X,Y\right)_{L^2} = \int_M g_{f_0(p)}(X(p), Y(p)) \vol_g$$
for $X,Y \in \Gamma(f_0^* TM)$.

For future reference we also quote a standard parabolic estimate and prove an interior estimate following from this.
\begin{thm}
Suppose $A_t$ is an elliptic differential operator of order $m$, uniformly elliptic in $t$, with $\mc{C}^{\infty}$ coefficients in $x$ and $t$.
Then for any $s \in \R$ and $T > 0$, there exists $C > 0$ such that
$$\|u_t\|_{H^s}^2 + \int_0^T \|u_{t'}\|_{H^{s+m'}}^2 dt' \leq C \left( \|u_0\|_{H^s}^2 + \int_0^T \|\partial_t u_{t'} - A_{t'} u_{t'}\|^2_{H^{s-m'}} dt' \right)$$
for any $t \in [0, T]$ and $u \in C^1([0,T], H^s) \cap C^0([0,T], H^{s+m'})$, where $m' = m/2$.
\end{thm}
For a proof, see 6.5.2 in \cite{Chazarain1982}.
We will need the following estimate for solutions, derived from this inequality:
\begin{cor}
\label{PE}
For any $\delta > 0$ and any 
 $A_t$ as above, there exists $C, \tilde{C}>0$, such that for any $u_t$ a solution of
$$\partial_t u_t = A_t u_t,$$
we have
$$\int_{\delta}^T \|u_{\tau}\|_{H^r}^2d \tau \leq C \int_0^T \|u_{\tau}\|_{H^{s}}^2 d\tau,$$
as well as
$$\|u_{t}\|_{H^r}^2 \leq \tilde{C} \int_0^T \|u_{\tau}\|_{H^{s}}^2 d\tau$$
for any $r,s \in \R$ and any $t \in [\delta, T]$.
\end{cor}
\begin{proof}
For $r < s$ the inequality is trivial.
For $r > s$ the claim follows inductively from
$$\int_{\delta}^T \|u_{\tau}\|^2_{H^{s+m'}} d\tau \leq C \int_0^T \|u_{\tau}\|_{H^{s-m'}}^2 d\tau.$$
For this consider $f:[0,T] \to [0,1]$ smooth such that $f(0) = 0, f(\delta) = 1$.
Then
$$\partial_t (f(t) u_t) - A_t u_t = (\partial_t f(t)) u_t.$$
Hence the above estimate yields
$$\int_{\delta}^T \|u_{\tau}\|_{s+m'}^2 d\tau \leq C \int_0^T \|u_{\tau}\|_{s-m'}^2 d\tau,$$
where $C = \max |\partial_t f|$.

We have shown that
$$\int_{\delta}^T \|u_{\tau}\|_{H^r}^2 d\tau \leq \tilde{C} \int_0^T \|u_{\tau}\|_{H^s}^2 d\tau.$$
Since $\partial_t u_t = A_t u_t$ and $u_t$ is a differential operator of order $m$ this implies
$$\int_{\delta}^T \|\partial_{\tau} u_{\tau}\|^2_{H^{r-m}} d\tau \leq  \tilde{C} \int_0^T \|u_{\tau}\|^2_{H^s} d\tau$$
and hence by the Sobolev embedding $W^{1,2}([a,b]; H^{l+1}, H^l) \hookrightarrow C^0([a,b]; H^l)$ (cf. \cite{Cherrier12}, Theorem 1.7.4 and (1.7.62))  we conclude
$$\|u_t\|_{H^{r-m}} \leq \hat{C} \int_0^T \|u_{\tau}\|^2_{H^s} d\tau.$$
(Here
$$W^{1,2}([a,b]; H^{l+1}, H^l) = L^2([a,b]; H^{l+1}) \cap \{u: [a,b] \to H^l : \partial_t u \in L^2([a,b]; H^l) \}$$
with the obvious norm.)
\end{proof}

\section{The Łojasiewicz inequality and gradient estimates}
The Łojasiewicz inequality relates the norm of the gradient of a differentiable function to its value near a critical point in a way that allows us to show convergence of the gradient flow. There are two situations when Łojasiewicz inequalities are known to hold. The optimal situation is when the function is a Morse function or less restrictively a Morse--Bott function. Then we have
$$|f(x) - f(x_0)| \leq C \|\grad f(x)\|^2$$
for $x_0$ a critical point of $f$ and some constant $C > 0$. This can be easily seen by applying the Morse--Bott lemma: near a critical manifold we may write a Morse--Bott function as
$$f(x_1, ..., x_n) = c + x_1^2 + ... + x_r^2 - x_{r+1}^2 - ... - x_s^2,$$
where $(x_1, ..., x_n)$ are coordinates with $x_0$ at the origin and critical manifold $\{x_{s+1} = ... = x_{n} = 0\}$.
Because in a small neighborhood the Riemannian metric is very close to being Euclidean, we get the inequality
$$|f(x) - c| \leq C |\grad f(x)|^2$$
for some $C > 0$.
The other case is that $f$ is analytic. Then there exists $\theta \in (1,2)$, such that
$$|f(x) - f(x_0)| \leq \|\grad f(x)\|^{\theta}.$$
We will make use of both versions. The inequality for analytic functions is a difficult theorem in the theory of semianalytic sets, due to Łojasiewicz. The first version will be employed to demonstrate stability of parallel spinors, since there we know $\mc{E}$ to be Morse--Bott. For volume constrained critical points we do not know this and instead use the weaker inequality for analytic functions.
Both inequalities are known in this general form only for functions on finite dimensional domains. We will spend most of the rest of the section justifying these inequalities for the spinor energy functional.

\begin{prop}[Optimal Łojasiewicz inequality for parallel spinors]
\label{LIP}
Let $\bar{\Phi}$ be a critical point of $\mc{E}$. (Hence $\bar{\Phi}$ is an absolute minimiser with $\mc{E}(\bar{\Phi}) = 0$.) Then there exists a $C^{2,\alpha}$ neighborhood $U$ of $\bar{\Phi}$ and some constant $C>0$, such that for any $\Phi \in U$ we have
$$\mc{E}(\Phi) \leq C \|Q(\Phi)\|_{L^2}^2.$$
\end{prop}

\begin{prop}[Łojasiewicz inequality for volume constrained critical points]
\label{LIVC}
Let $\bar{\Phi} = (\bar{g},\bar{\varphi})$ be a volume constrained critical point of $\mc{E}$. Then there exists a $C^{2,\alpha}$ neighborhood $U$ of $\bar{\Phi}$ and some constant $\theta \in (1,2)$, such that for any $\Phi=(g,\varphi)$ with $\int_M \vol_g = \int_M \vol_{\bar{g}}$ we have
$$|\mc{E}(\Phi) - \mc{E}(\bar{\Phi})| \leq \|\mathring{Q}(\Phi)\|_{L^2}^{\theta}.$$
If the set of volume constrained critical points near $\bar{\Phi}$ is a manifold, this can be improved to
$$|\mc{E}(\Phi) - \mc{E}(\bar{\Phi})| \leq C \|\mathring{Q}(\Phi)\|_{L^2}^2.$$
\end{prop}
The proofs of both propositions rely on the following infinite-dimensional form of the Łojasiewicz inequality, due to Colding and Minicozzi II, see \cite{Colding2013}.
\begin{thm}
\label{CML}
\begin{enumerate}
\item Suppose $E \subset L^2$ is a closed subspace, $U$ is an open neighborhood of $0$ in $C^{2, \beta} \cap E$.
\item Suppose $G: U \to \R$ is an analytic function or that there is a neighborhood $V$ of $0$, such that $\{x \in V : \grad G (x) = 0\}$ is a finite dimensional submanifold.
\item Suppose the gradient $\grad G : U \to C^{\beta} \cap E$ is $C^1$, $\grad G (0) = 0$ and
$$\|\grad G(x) - \grad G(y) \|_{L^2} \leq C \|x - y\|_{H^2}$$
\item $L = D \grad G (0)$ is symmetric, bounded from $C^{2, \beta} \cap E$ to $C^{\beta} \cap E$ and from $H^2 \cap E$ to $L^2 \cap E$ and Fredholm from $C^{2, \beta} \cap E$ to $C^{\beta} \cap E$.
\end{enumerate}
Then there exists $\theta \in (1,2)$ so that for all $x\in E$ sufficiently small
$$|G(x) - G(0)| \leq \|\grad G(x)\|_{L^2}^\theta$$
If there is a neighborhood $V$ of $0$, such that $\{x \in V : \grad G (x) = 0\}$ is a finite dimensional submanifold, we get the stronger inequality
$$|G(x) - G(0)| \leq C \|\grad G(x)\|_{L^2}^2$$
for some $C>0$.
\end{thm}
{\em Remark.} Colding and Minicozzi II prove this for $G$ analytic. The alternative condition we give is essentially that $G$ is Morse--Bott at $0$. The proof in that case is the same except that when the finite dimensional Łojasiewicz inequality is used, we instead invoke the stronger inequality for Morse--Bott functions.

Since this theorem requires the linearisation of the gradient to be Fredholm we will be working on a slice of the spin diffeomorphism group. 
\begin{lemma}
Let $\bar{\Phi} = (\bar{g}, \bar{\varphi})$ be a critical point. Let $\iota: \ker \lambda_{\bar{\Phi}} \to \Gamma(\Sigma M)$ be the inclusion. $f = \mc{E} \circ \Xi_{\bar{\Phi}} \circ \iota$ fulfills the conditions of theorem \ref{CML}.
In particular we have
$$|f(x)| \leq C \|\grad f(x)\|_{L^2}^2$$
\end{lemma}
\begin{proof}
We equip $\Gamma(\odot^2 T^* M) \oplus \Gamma(\Sigma_g M)$ with the $L^2$ metric induced by $\bar{g}$, and similarly we define the $C^{2,\alpha}$ norm in terms of $\bar{g}$.
Then clearly $\mc{E} \circ \Xi_{\bar{\Phi}} \circ \iota$ is a smooth function and by \cite{Ammann2015b} its critical set is smooth, thus the second condition in theorem \ref{CML} is fulfilled.
Moreover $0$ corresponds to $\bar{\Phi}$ and hence is a critical point, i.e. $\grad f(0) = 0$.
The gradient of $f$ can be considered as a nonlinear second order differential operator. In fact, it is a smooth map 
$$\grad f: \Gamma^{2,\alpha}(\odot^2 T^*M \oplus \Sigma_g M) \to \Gamma^{\alpha}(\odot^2 T^*M \oplus \Sigma_g M).$$
On any bounded $C^{2,\alpha}$ neighborhood $U$ of $0$ we have
$$\|\grad f(x) - \grad f(y)\|_{L^2} \leq C \|x - y\|_{H^2}.$$
This is a simple consequence of the fact that $Q(g,\varphi)$ can be locally represented as a polynomial expression
in the coordinate expressions of $g$ and $\varphi$ and their first and second derivatives. In a bounded $C^{2,\alpha}$ neighborhood we then estimate terms as needed to get an expression which is bounded by $\|(g,\varphi)\|_{H^2}$.
This concludes the argument for conditions 1,2 and 3.

Since $DQ(\bar{\Phi})$ is symmetric (by \cite{Ammann2015}), so is $L$. Since $L$ is a linear second order differential operator, it induces continuous maps $C^{2,\alpha} \to C^{\alpha}$ and $H^2 \to L^2$.
It remains to be shown that $L$ is Fredholm.
To see this, remember that we have a splitting
$$T_{\bar{\Phi}} \mc{N} = \ker \lambda_{\bar{\Phi}} \oplus \im \lambda^*_{\bar{\Phi}}.$$
With respect to these operators, we know the two identities
$$DQ(\bar{\Phi}) \circ \lambda^*_{\bar{\Phi}} = 0 \text{ and } \lambda_{\bar{\Phi}} \circ DQ(\bar{\Phi}) = 0,$$
both of which reflect diffeomorphism invariance of $Q$.
Moreover, we introduced the perturbed gradient $\tilde{Q}_{\bar{\Phi}}$, which we know is strongly elliptic and thus its linearization is Fredholm. Its linearization is also symmetric.
Thus we conclude that $DQ(\bar{\Phi})$ has the form
$$\bordermatrix{
                      & \ker \lambda_{\bar{\Phi}} & \im \lambda_{\bar{\Phi}}^* \cr
\ker \lambda_{\bar{\Phi}} & P                   & 0                      \cr
\im \lambda_{\bar{\Phi}}^* & 0                   & 0                      \cr
},$$
whereas $D\tilde{Q}_{\bar{\Phi}}(\bar{\Phi})$ has the form 
$$\bordermatrix{
                      & \ker \lambda_{\bar{\Phi}} & \im \lambda_{\bar{\Phi}}^* \cr
\ker \lambda_{\bar{\Phi}} & P                   & 0                      \cr
\im \lambda_{\bar{\Phi}}^* & 0                   & R                      \cr
}.$$
Since $D\tilde{Q}_{\bar{\Phi}}(\bar{\Phi})$ is Fredholm, so is $P = \pi \circ DQ(\Phi) \circ \iota$, where $\pi: T_{\bar{\Phi}} \mc{N} \to \ker \lambda_{\bar{\Phi}}$ denotes the orthogonal projection.
We compute
$$D\grad f (0) = D(\Xi \circ \iota)(0)^* DQ(\Xi \circ \iota(x)) = \pi \circ D\Xi(0)^* DQ(\bar{\Phi}).$$
Since the domain is restricted to $\ker \lambda_{\bar{\Phi}}$ and $D\Xi(0) = \id$, we conclude that
$$D \grad f(0) = P,$$
and hence $L = D\grad f(0)$ is Fredholm as required. Thus we have checked all conditions in theorem \ref{CML}, and the inequality holds.
\end{proof}

\begin{proof}[Proof of proposition \ref{LIP}]
What remains to be shown is that the inequality
$$|f(x)| \leq \|\grad f(x)\|_{L^2}^2$$
implies the inequality
$$|\mc{E}(\Phi)| \leq C \|Q(\Phi)\|^2_{L^2}.$$
First, by the slice theorem there exists a $C^{k+1,\alpha}$ neighborhood $U$ of $\bar{\Phi}$, such that for any $\Phi \in U$ there exists a diffeomorphism $f: M\to M$, such that
$$\lambda_{\bar{\Phi}}(\Xi^{-1}(F_* \Phi)) = 0.$$
Since
$$\mc{E}(F_* \Phi) = \mc{E}(\Phi), \quad F_* Q(\Phi) = Q(F_* \Phi)$$
and since the $L^2$ metric is diffeomorphism invariant, we can assume that $\Phi$ lies in the slice, i.e. $\lambda_{\bar{\Phi}}(\Xi^{-1}(\Phi)) = 0$. Then we have
$$|\mc{E}(\Phi)| = f(\Xi^{-1}(\Phi)) \leq \|\grad f(\Xi^{-1}(\Phi))\|_{L^2}^2.$$
Hence we must show
$$\|\grad f (\Xi^{-1}(\Phi))\|_{L^2}^2 \leq \|Q(\Phi)\|_{L^2}.$$
First note that the metric on $\ker \lambda_{\bar{\Phi}}$ is the metric induced by $\bar{\Phi}$. By making the neighborhood smaller if necessary, we can assume that all $L^2$ metrics in that neighborhood are uniformly equivalent. We have
$$\grad f(\Xi^{-1}(\Phi)) = D(\Xi \circ \iota) (\Xi^{-1}(\Phi))^* Q(\Phi).$$
Since $D(\Xi \circ \iota)$ is clearly Lipschitz, we obtain our estimate. This concludes the proof of the Łojasiewicz inequality in this case.
\end{proof}
\begin{proof}[Proof of proposition \ref{LIVC}]
For the purposes of the following discussion, read the spaces of smooth mappings as the spaces of $C^{2,\alpha}$ mappings, so that they are Banach spaces or Banach manifolds.
By the analytic regular value theorem, we can find an analytic parametrization of
$$\left\{ g \in \Gamma(\odot^2_+ T^*M) : \int_M \vol_g = 1\right\}$$
by
$$\left\{ h \in \Gamma(\odot^2_+ T^* M) : \int_M \tr_g h \vol_g = 0\right\}.$$
(For a treatment of the implicit function theorem in the analytic category on Banach spaces, take for example \cite{Hajek14}, theorem 174.)
We combine this parametrization with $\Xi_{g,\varphi}$ to obtain an analytic parametrization
$$\Psi: U \subset V_0 \to \mc{N}^1$$
where
$$V_0 = \left\{(h, \psi) \in \ker \lambda_{g,\varphi} : \int_M \tr_g h \vol_g = 0\right\}$$
and
$$\mc{N}^1 = \left\{\Phi = (g,\varphi) \in \mc{N} : \int_M \vol_g = 1\right\}.$$
Define $f = \mc{E} \circ \iota \circ \Psi$, with $\iota: \mc{N}^1 \to \mc{N}$ the inclusion.
Then $f$ fulfills the conditions of theorem \ref{CML}, which can be shown as in the previous lemma. 
Applying the theorem, we thus obtain
$$|f(x) - f(0)| \leq \|\grad f (x)\|_{L^2}^{\theta},$$
where $\theta \in (1,2)$. If the critical set is a manifold near $\bar{\Phi}$, we use the optimal version theorem of theorem \ref{CML} and obtain
$$|f(x) - f(0)| \leq C \|\grad f(x)\|_{L^2}^2$$
for some $C > 0$.
 What remains to be shown is
$$\|\grad f(\Psi^{-1}(\Phi))\|_{L^2} \leq C \|\mathring{Q}(\Phi)\|_{L^2}.$$
As in the previous proposition, we compute
$$\grad f(\Psi^{-1}(\Phi)) = (D\Psi )^*(\Psi^{-1}(\Phi)) \grad (\mc{E} \circ \iota)(\Phi).$$
Then the claim follows, since, on the one hand, $D\Psi$ is Lipschitz by the regular value theorem, and on the other hand
$$\grad(\mc{E} \circ \iota)(\Phi) = D\iota(\Phi)^* \grad \mc{E}(\Phi) = D\iota(\Phi)^* Q(\Phi).$$
Since $D\iota(\Phi)^* : T_{\Phi} \mc{N} \to T_{\Phi} \mc{N}^1$ is the orthogonal projection,
this implies
$$\grad(\mc{E} \circ \iota)(\Phi) = \mathring{Q}(\Phi).$$
\end{proof}

\begin{thm}[Energy decay]
\label{L2cv}
Suppose $M$ is a compact manifold.
\begin{enumerate}
\item Suppose $\bar{\Phi}$ is a critical point of $\mc{E}$. Then there exists a $C^{2,\alpha}$ neigborhood $U$ of $\bar{\Phi}$, such that for any $\Phi \in U$ the following inequalities hold
$$\mc{E}(\Phi_t) \leq C e^{-\alpha t},$$
$$\int_T^{\infty} \|Q(\Phi_t)\|^2_{L^2} dt \leq C e^{-\alpha T},$$
and
$$\int_T^{\infty} \|Q(\Phi_t)\|_{L^2} dt \leq C e^{-\alpha T},$$
where $C, \alpha>0$ and $\Phi_t$ is the solution of
$$\partial_t \Phi_t = Q(\Phi_t), \Phi_0 = \Phi.$$
\item Suppose $\bar{\Phi}$ is a volume constrained minimizer of $\mc{E}$. Then there exists a $C^{2,\alpha}$ neighborhood $U$ of $\bar{\Phi}$, such that for any $\Phi \in U$ it holds
$$|\mc{E}(\Phi_t) - \mc{E}(\bar{\Phi})| \leq \frac{C}{1+T^{\beta}},$$
$$\int_T^{\infty} \|\mathring{Q}(\Phi_t)\|^2_{L^2} dt \leq \frac{C}{1 + T^{\beta}}$$
and
$$\int_T^{\infty} \|\mathring{Q}(\Phi_t)\|_{L^2} dt \leq \frac{C}{1+T^{\gamma}},$$
for some $C, \beta > 1$. If the set of volume constrained critical sets is a manifold near $\bar{\Phi}$, we can instead choose exponential bounds as in the first case. Here we assume $\Phi_t$ is the volume normalized spinor flow with initial condition $\Phi_0 = \Phi \in U$.
\end{enumerate}
The integrals are to be read as the integral from $T$ to the maximal time of existence in the neighborhood $U$. The constants $C, \alpha, \beta$ only depend on the constants $C$ and $\theta$ in the Łojasiewicz inequalities.
\end{thm}
{\em Remark. } The constants $\beta$ and $\gamma$ can be computed from the constant $\theta$ in the Łojasiewicz inequality as $\beta = \frac{\theta }{2 - \theta}$ and $\gamma = \frac{\theta - 1}{2 - \theta}$. As $\theta$ tends to $2$, $\beta$ tends to infinity, i.e. the convergence rate improves. As $\theta$ tends to $1$, $\beta$ tends to $1$, i.e. the convergence rate gets worse. Likewise, $\gamma$ tends to $\infty$ if $\theta$ tends to $2$, but $\gamma$ tends to $0$ as $\theta$ tends to $1$. 
\begin{proof}
First we note that
$$\frac{d}{dt} \mc{E}(\Phi_t) =  -\|Q(\Phi_t)\|_{L^2}^2$$
implies that the integral of the gradient over all future time is controlled by the energy at a fixed time.
Now applying the optimal Łojasiewicz inequality, we obtain 
$$\frac{d}{dt} \mc{E}(\Phi_t) \leq -\frac{1}{C} \mc{E}(\Phi_t).$$
Integrating this differential inequality, we obtain
$$\mc{E}(\Phi_t) \leq \mc{E}(\Phi_0) e^{-(1/C) t}.$$
Choosing the neighborhood so that $\mc{E}$ is bounded, we obtain the desired inequality.

For the second case consider
$$\frac{d}{dt} |\mc{E}(\Phi_t) - \mc{E}(\bar{\Phi})| = -\|Q(\Phi_t)\|_{L^2}^2 \leq -|\mc{E}(\Phi_t) - \mc{E}(\bar{\Phi})|^{2/\theta}.$$
Integrating this differential inequality, we obtain
$$|\mc{E}(\Phi_t) - \mc{E}(\bar{\Phi})| \leq \left(\frac{2}{\theta} - 1\right) \frac{1}{(C + t)^{\beta}}$$
where $\beta = \frac{1}{2/\theta - 1}$ and $C = |\mc{E}(\Phi_0) - \mc{E}(\bar{\Phi})|^{1-2/\theta}$.
By continuity of $\mc{E}$ we can find a lower bound for $C$ on a small neighborhood, and using this lower bound we obtain the desired inequality. The bound for the integral of $\|\mathring{Q}(\Phi_t)\|_{L^2}$ follows as above.

For the estimates of $\int_T^{\infty} \|Q(\Phi_t)\| dt$, notice that the Łojasiewicz inequality implies $\mc{E}(\Phi)^{-1/\theta} \geq C \|Q(\Phi)\|^{-1}$. (Here we actually have $\theta = 2$. The case of volume constrained minimizers is analogous with $\theta \neq 2$ in general.)
This implies
\begin{eqnarray*}
-\frac{d}{dt}\mc{E}(\Phi_t)^{1-1/\theta} & = & (1-1/\theta) \mc{E}(\Phi_t)^{-1/\theta} \|Q(\Phi_t)\|^2\\
& \geq & C \|Q(\Phi_t)\| 
\end{eqnarray*}
Hence
$$\int_T^{\infty} \|Q(\Phi_t)\|_{L^2} dt \leq C \mc{E}(\Phi_T)^{1-1/\theta}.$$
Plugging in the estimate for $\mc{E}(\Phi_T)$ then gives the desired result.
\end{proof}

\section{Mapping flow estimates}
Suppose $\Phi_t$ solves
$$\partial_t \Phi_t = Q(\Phi_t).$$
In the previous section we proved a strong estimate of the gradient along the flow in the $L^2$ norm, provided $\Phi_t$ is near a critical point. We would now like to improve this to an estimate in some higher regularity norm. Since the gradient $Q_t = Q(\Phi_t)$ satisfies the linear parabolic equation
$$\partial_t Q_t = DQ(\Phi_t) Q_t,$$
this is reasonable by parabolic regularity. Unfortunately, this equation is only weakly parabolic and hence we can not directly apply parabolic regularity. However, we recall that $\tilde{\Phi}_t = F_{t*} \Phi_t$ obeys the strongly parabolic equation
$$\partial_t \tilde{\Phi}_t = \tilde{Q}(\tilde{\Phi}_t)$$
if $f_t$ satisfies the mapping flow equation
$$\partial_t f_t = P_{g_t, g_0}(f_t), f_0 = \id_M.$$
The gauged gradient $\tilde{Q}_t = \tilde{Q}(\tilde{\Phi}_t)$ satisfies the linear strongly parabolic equation
$$\partial_t \tilde{Q}_t = D\tilde{Q}(\tilde{\Phi}_t) \tilde{Q}_t.$$
Parabolic regularity applies to $\tilde{Q}_t$, but we have no estimate of $\tilde{Q}_t$! To obtain such an estimate, we will now show how to control $\partial_t f_t$ along the mapping flow. In the next section, we will combine this estimate with the gradient estimate of the previous section to obtain an estimate of $\tilde{Q}_t$.

\begin{lemma}
\label{MFE}
Let $\tilde{g} \in \Gamma(\odot^2_+ T^* M)$ and $k > \frac{n}{2} + 2$. Suppose $\tilde{g}$ has no Killing fields.
Then there exists a $H^k$ neighborhood $U \times V$ of $(\id_M, \tilde{g})$ and constants $C, \lambda > 0$, such that for a solution $f_t$ and a metric $g_t \in V$, $g_t$ once differentiable in time, of an initial value problem
$$f_0 = \id_M$$
$$\dot{f_t} = P_{g_t, \tilde{g}}(f_t)$$
we have
$$\int_{t_1}^{t_2} \|P_{g_t, \tilde{g}}\|_{H^{-2}} dt \leq C \left( \int_0^{t_1} \|\dot{g}_t\|_{L^2} e^{\lambda (t - t_1)} dt + \int_{t_1}^{t_2} \|\dot{g}_t\|_{L^2} dt  + e^{-\lambda t_1}\right)$$
for some $C, \lambda > 0$, provided the flow exists until time $t_2$ in the neighborhood $U \times V$.
\end{lemma}
\begin{proof}
As computed in \cite{Ammann2015},
$$DP_{\tilde{g}, \tilde{g}}(\id_M) X = -4 (\delta_{\tilde{g}} \delta_{\tilde{g}}^* X^{\flat})^{\sharp}.$$
A computation of the symbol then shows that this operator is strongly elliptic. 
Furthermore, this formula implies
$$\left(DP_{\tilde{g}, \tilde{g}}(\id_M) X, X\right)_{L^2} = -4 \left(\delta^*_{\tilde{g}} X^{\flat}, \delta^*_{\tilde{g}} X^{\flat} \right) = -4 \left(\mc{L}_X \tilde{g}, \mc{L}_X \tilde{g} \right)_{L^2}.$$
Since we assume $\tilde{g}$ has no Killing fields, this implies $DP_{\tilde{g}, \tilde{g}}(\id_M)$ is strictly negative definite, i.e. there exists $\mu > 0$, such that
$$\left(DP_{\tilde{g}, \tilde{g}}(\id_M) X, X\right)_{L^2} \leq -\mu \left(X,X\right)_{L^2}.$$
Since the coefficients of the operator $P_{g_1, g_2}(f)$ are continuous in $f$ and the first derivatives of $g_1$ and $g_2$ and recalling that by the Sobolev embedding theorem $H^k$ continuously embeds in $C^2$, we conclude that there is a $H^k$ neighborhood $U$ of $\tilde{g}$, a neighborhood $V$ of $\id_M$ and a constant $0 < \lambda < \mu$, such that $DP_{g, \tilde{g}}(f)$ is strongly elliptic and strictly negative definite with a constant $\lambda$.

Since $L = DP_{\tilde{g}, \tilde{g}}(\id_M)$ is strictly negative definite, it induces an invertible operator from $H^{s+2} \to H^s$. We have, up to equivalence,
$$\|f\|_{H^{-2}} = \|L^{-1}f\|_{L^2}.$$
This implies, in particular, that $DP_{g, \tilde{g}}(f)$ is also strictly negative definite with respect to the Sobolev inner product $\langle \cdot, \cdot \rangle_{H^{-2}}$.

We will now derive a differential inequality for $\|\dot{f}_t\|_{H^{-2}}^2$, where
$$\dot{f}_t = P_{g_t, \tilde{g}}(f_t).$$
For brevity, we let $P_{g_t, \tilde{g}}(f_t) = P_{g_t}(f_t)$.
In what follows, we tacitly assume $g_t \in U$, $f_t \in V$ for all $t$, as per the statement of the lemma. 
We calculate
\begin{eqnarray*}
\frac{1}{2} \frac{d}{dt} \langle P_{g_t}(f_t), P_{g_t}(f_t) \rangle_{H^{-2}} & = &  \langle \frac{d}{dt} P_{g_t}(f_t), P_{g_t}(f_t) \rangle_{H^{-2}}\\
& = & \langle P_{\dot{g_t}}(f_t) + DP_{g_t}(f_t) \dot{f}_t, P_{g_t}(f_t) \rangle_{H^{-2}} \\
& = & \langle P_{\dot{g_t}}(f_t), P_{g_t}(f_t) \rangle + \langle DP_{g_t}(f_t) P_{g_t}(f_t), P_{g_t}(f_t) \rangle_{H^{-2}}.
\end{eqnarray*}
The map
$$g \mapsto P_g(f) = 2 df(\delta_{f^*\tilde{g}} g),$$
is a linear first order differential operator with bounds dependent on $\|f\|_{C^1}$ and $\|\tilde{g}\|_{C^1}$. As such we can estimate, using that bound and the Cauchy-Schwarz inequality
$$|\langle P_{\dot{g_t}}(f_t), P_{g_t}(f_t) \rangle_{H^{-2}}| \leq \|P_{\dot{g_t}} (f_t)\|_{H^{-2}} \|P_{g_t}(f_t)\|_{H^{-2}}  \leq  C \|\dot{g}_t\|_{L^2} \|P_{g_t}(f_t)\|_{H^{-2}}.$$
Then we obtain for
$$a(t) = \langle P_{g_t}(f_t), P_{g_t}(f_t) \rangle_{H^{-2}}$$
the inequality
$$\frac{1}{2} \dot{a}(t) \leq C \|\dot{g}_t\|_{L^2} \sqrt{a(t)} - \lambda a(t).$$
Let $b(t) = \sqrt{a(t)}$. The function $b$ then satisfies the following differential inequality
$$\dot{b}(t) \leq -\lambda b(t) + \|g_t\|_{L^2}.$$
Define
$$\beta(t) = e^{-\lambda t} \left( b(0) + \int_0^t e^{\lambda s} \|\dot{g}_s\|_{L^2} ds \right).$$
Then we have
$$\dot{\beta}(t) = -\lambda \beta(t) + \|\dot{g}_t\|_{L^2}.$$
We deduce
$$\frac{d}{dt} (b-\beta) \leq -\lambda (b-\beta),$$
and since $b(0) = \beta(0)$, $b(t) \leq \beta(t)$ follows.
To obtain the claim of the lemma, we will now estimate the integral of $\beta(t)$. For brevity, we denote $\gamma(t) = \|\dot{g}_t\|_{L^2}$. Define $\chi(s,t) = 1$ if $0 \leq s \leq t$ and $\chi(s,t) = 0$ otherwise. Then we calculate
\begin{eqnarray*}
\int_{t_1}^{t_2} e^{-\lambda t} \int_0^t e^{\lambda s} \gamma(s) ds dt
& = & \int_{t_1}^{t_2} \int_0^t e^{\lambda(s-t)} \gamma(s) ds dt \\
& = & \int_{t_1}^{t_2} \int_0^{t_2} \chi(s,t) e^{\lambda(s-t)} \gamma(s) ds dt\\
& = & \int_0^{t_2} \gamma(s) \int_{t_1}^{t_2} \chi(s,t) e^{\lambda(s-t)} dt ds\\
& = & \int_0^{t_2} \gamma(s) \int_{\max\{s, t_1\}}^{t_2} e^{\lambda(s-t)} dt ds\\
& = & \int_0^{t_1} \gamma(s) \int_{t_1}^{t_2} e^{\lambda(s-t)} dt ds + \int_{s}^{t_2} \gamma(s) \int_{t_1}^{t_2} e^{\lambda (s-t)} dt ds\\
& \leq & \lambda^{-1} \left( \int_0^{t_1} e^{\lambda (s-t_1)}\gamma(s) ds + \int_{t_1}^{t_2} \gamma(s) ds \right)   
\end{eqnarray*}
The integral of the term $b(0) e^{-\lambda t}$ is
$$\int_{t_1}^{t_2} b(0) e^{-\lambda t} dt = \lambda^{-1} b(0) \left( e^{-\lambda t_1} - e^{-\lambda t_2} \right).$$
Thus
\begin{eqnarray*}
  \int_{t_1}^{t_2} \beta(t) dt & \leq & \lambda^{-1} \left( b(0) e^{-\lambda t_1} +  \int_0^{t_1} e^{\lambda (s-t_1)}\gamma(s) ds + \int_{t_1}^{t_2} \gamma(s) ds \right) 
\end{eqnarray*}
and the claim of the lemma follows.
\end{proof}

\section{Smooth convergence of the flow}
Now everything is in place to prove stability of the spinor flow. We obtain slightly sharper theorems than in the introduction:
\begin{thm}
\label{stabP}
Suppose $\bar{\Phi} = (\bar{g}, \bar{\varphi})$ is a critical point of $\mc{E}$, such that $\bar{g}$ has no Killing fields. Then for any $k > \frac{n}{2} + 5$ there exists a $H^k$ neighborhood $U$ of $\bar{\Phi}$, such that any solution of the negative gradient flow $\Phi_t$ with initial condition $\Phi_0 = \Phi \in U$ converges in $H^k$ to a critical point. The speed of convergence is exponential.
\end{thm}
\begin{thm}
\label{stabVC}
Suppose $\bar{\Phi} = (\bar{g}, \bar{\varphi})$ is a volume constrained minimizer of $\mc{E}$ and suppose the set of critical points is a manifold near $\bar{\Phi}$. Suppose furthermore, that $\bar{g}$ has no Killing fields and $k > \frac{n}{2} + 5$. Then there exists a $H^k$ neighborhood $U$ of $\bar{\Phi}$, such that a solution of the volume constrained negative gradient flow $\Phi_t$ with initial condition $\Phi_0 = \Phi \in U$ converges in $H^k$ to a critical point. The speed of convergence is exponential.

If the critical set is not a manifold, but $\theta$ in proposition \ref{LIVC} can be chosen to be larger than $3/2$, then there exists a $H^k$ neighborhood $U$ of $\bar{\Phi}$, such that a solution of the volume constrained negative gradient flow $\Phi_t$ with initial condition $\Phi_0 = \Phi \in U$ converges in $H^k$ to a critical point. The speed of convergence is $O(T^{-\kappa})$, $\kappa = \frac{2\theta -3}{2-\theta} >0$.
\end{thm}

We will reduce the proof of these theorems to the following two lemmas:
\begin{lemma}[Existence near critical points]
\label{UniformExistence}
Let $\bar{\Phi}$ be a critical point of $\mc{E}$ and let $T, \epsilon > 0, k > \frac{n}{2} + 2$.
Then there exists $\delta > 0$, such that for any $\Phi$ with $\|\Phi - \bar{\Phi}\|_{H^k} < \delta$, the flow
$$\partial_t \Phi_t = \tilde{Q}(\Phi_t), \Phi_0 = \Phi$$
exists until time $T$ and $\|\Phi_T - \bar{\Phi}\|_{H^k} < \epsilon$. The same result holds for volume constrained critical points and the volume constrained flow. 
\end{lemma}
The proof is analogous to the proof of corollary 8.6 in \cite{Weiss2012}
\begin{lemma}[Decay of the gradient in a Sobolev norm]
\label{GradientDecay}
Suppose $\bar{\Phi}$ is a critical point of $\mc{E}$. Then for any $k > \frac{n}{2} + 5$ there exists a $H^k$ neighborhood $U$ of $\bar{\Phi}$, a neighborhood $V$ of $\id_M$ in $\Diff(M)$, constants $C, \alpha > 0$, such that for $\Phi \in U$ the gauged spinor flow $\tilde{\Phi}_t$ with initial condition $\Phi$ fulfills the following estimate 
\begin{equation}
\label{GEa}
\|\tilde{Q}(\tilde{\Phi}_t)\|_{H^k} \leq C e^{-\alpha T}
\end{equation}
as long as $\Phi_t$ and $f_t$ remain in the neighborhoods $U$ and $V$ respectively.

Analogously, if $\bar{\Phi}$ is a volume constrained critical point of $\mc{E}$ and the critical set near $\bar{\Phi}$ is a manifold, then for any $k > \frac{n}{2} + 5$ there exists a $H^k$ neighborhood $U$ of $\bar{\Phi}$, a neighborhood $V$ of $\id_M$ in $\Diff(M)$, constants $C, \alpha > 0$, such that for $\Phi \in U$ the volume normalized gauged spinor flow $\tilde{\Phi}_t$ with initial condition $\Phi$ fulfills the following estimate 
\begin{equation}
\label{GEb}
\|\mathring{\tilde{Q}}(\tilde{\Phi}_t)\|_{H^k} \leq C e^{-\alpha T}
\end{equation}
as long as $\Phi_t$ and $f_t$ remain in the neighborhoods $U$ and $V$ respectively. If the critical set is not a manifold we instead find $C, \beta > 0$, such that
\begin{equation}
\label{GEc}
\|\mathring{\tilde{Q}}(\tilde{\Phi}_t)\|_{H^k} \leq \frac{C}{1+ T^\beta}
\end{equation}
\end{lemma}
\begin{proof}[Proof of the lemma]
We start with the first case.
We will show this estimate by combining the gradient estimate from the Łojasiewicz inequality and the estimate of the mapping flow. This will give us an estimate of the time integral of $\|\tilde{Q}(\tilde{\Phi}_t)\|_{H^s}$ for $s=-3$, which we will then improve via parabolic regularity.
We consider the spinor flow
$$\partial_t \Phi_t = Q(\Phi_t), \Phi_0 = \Phi,$$
the gauged spinor flow
$$\partial_t \tilde{\Phi}_t = \tilde{Q}(\tilde{\Phi}_t), \tilde{\Phi}_0 = \Phi$$
and the mapping flow
$$\partial_t f_t = P_{g_t, \bar{g}}(f), f_0 = \id_M.$$
Then we have that
$$\tilde{\Phi}_t = F_t^* \Phi_t$$
and hence
\begin{eqnarray*}
\tilde{Q}(\tilde{\Phi}_t) & = & \partial_t (F_t^* \Phi_t) \\
& = & F_t^* \tilde{\mc{L}}_{X_t} \Phi_t + F_t^* \dot{\Phi}_t
\end{eqnarray*}
where $X_t = \frac{d}{dt} f_t$ and $\tilde{\mc{L}}$ is the spinorial Lie derivative.

Multiplication of Sobolev functions $H^k \times H^s \to H^s$ for negative $s$ and positive $k$ is continous, if $k > -s$ and $k > n/2$, where $n$ is the dimension of the manifold, see theorem 2 (i), sect. 4.4.3 in \cite{Runst96}. In particular, our choice of $k$ allows any $s \geq -3$.

We will use this to estimate $\tilde{\mc{L}}_{X_t} \Phi_t$ in the $H^s$ norm.
Recall that 
$$\tilde{\mc{L}}_X \Phi = (\mc{L}_X g, \tilde{\mc{L}}_X \varphi) = (2 \delta_g^* X^{\flat}, \nabla_X^g \varphi - \frac{1}{4} dX^{\flat} \cdot \varphi).$$
In local coordinates we have
$$\mc{L}_X g = p_1(g_{jk}, \partial_l g_{mn}, X^i) +  p_2(g_{ij}, \partial_k X^l)$$
for some polynomials $p_1, p_2$, which are linear in the partial derivative terms and the $X^i$ terms. Likewise we have
$$\tilde{\mc{L}}_X \varphi = q_1(X^i, \partial_j\varphi^{\alpha}) + q_2(g_ij, \partial_l g_{mn}, X^k, \varphi^{\alpha})$$
for polynomials $q_1, q_2$, linear in the partial derivative terms and the $X^i$ terms.
From this follows, using the multiplication theorem above and the fact that $H^{k-1}$ is a Banach algebra (since it embeds into $C^2$),
\begin{eqnarray*}
\|\tilde{\mc{L}}_X \Phi\|_{H^s} & \leq & C \left( \|DX\|_{H^s} \sum_{d=1}^r \|\Phi\|_{H^{k-1}}^d +  \|X\|_{H^s} \sum_{d=1}^r \|D\Phi\|_{H^{k-1}}^d \right)\\
& \leq & C \left( \|X\|_{H^{s+1}} \sum_{d=1}^r \|\Phi\|_{H^{k-1}}^d + \|X\|_{H^s} \sum_{d=1}^r \|\Phi\|_{H^{k}}^d \right)\\
& \leq & C \left(\|X\|_{H^{s+1}} \sum_{d=1}^r \|\Phi\|_{H^{k}}^d\right)
\end{eqnarray*}
for $k > -s + n/2 + 2$, where $r$ is the maximal degree of the polynomials $p_1, p_2, q_1, q_2$.
Since we will choose $s=-3$ and $k > n/2 + 5$, this will be the case.

Furthermore, given a diffeomorphism $f: M \to M$ and a lift to the topological spin structure $F: \tilde{P} \to \tilde{P}$, we have
$$F^* \Phi = \Phi \circ F,$$
where we view $\Phi$ as an equivariant map $\Phi: \tilde{P} \to \left(\widetilde{\GL_n^+} \times \Sigma_n\right)/\Spin(n)$.
Using the transformation rule, we can derive an estimate
$$\|u \circ f\|_{W^{k,p}(M)} \leq \nu(\|f\|_{C^{\max \{k, 1\}}}) \|u\|_{W^{k,p}(M)}$$
for the integral Sobolev spaces.
For real $s$, we conclude the following inequality by interpolation and duality
$$\|F^* \Phi\|_{H^s} \leq \tilde{\nu} (\|F\|_{C^{\lceil|s|\rceil}}) \|\Phi\|_{H^s},$$
where $\nu, \tilde{\nu}: [0, \infty) \to [0, \infty)$ are continuous functions.
 
In conclusion we obtain
\begin{eqnarray*}
\|\tilde{Q}(\tilde{\Phi}_t)\|_{H^s} & = & \|F_t^* \tilde{\mc{L}}_{X_t} \Phi_t + F_t^* \dot{\Phi}_t\|_{H^s} \\
& \leq & C \nu(\|F_t\|_{C^{\lceil|s|\rceil}}) ( \|X_t\|_{H^{s+1}} \|\Phi_t\|_{H^k} +  \|\dot{\Phi}_t\|_{H^s} )
\end{eqnarray*}
We will assume both $f_t$ and $\Phi_t$ to remain in a bounded $H^k$ neighborhood, thus we can estimate their norms by a constant, hence we obtain
$$\|\tilde{Q}(\tilde{\Phi}_t)\|_{H^s} \leq C (\|\dot{f}_t\|_{H^{s+1}} + \|\dot{\Phi}_t\|_{H^s}).$$

It remains to choose a neighborhood of $\bar{\Phi}$ so that we can also estimate the terms $\|\dot{f}_t\|_{H^{s+1}}$ and $\|\dot{\Phi}_t\|_{H^s}$.

By theorem \ref{L2cv} there exists a $H^k$ neighborhood $U$ of $\bar{\Phi}$, such that for any $\Phi \in U$ it holds
$$\int_T^{T_{\max}} \|Q(\Phi_t)\|_{L^2} dt \leq C e^{-\alpha T}.$$

Choose a neighborhood $U\times V_m$ of $(\id_M, \bar{g})$ such that we have the mapping flow estimate \ref{MFE}.
Choose a neighborhood $V_s$ of $\bar{\Phi}$, such that we have the $L^2$ estimate of the gradient along the spinor flow as in theorem \ref{L2cv}. We may assume that $\pi_{\Sigma}(V_s) = V_m$.
Furthermore, we choose the neighborhoods to be bounded in $H^k$.

Now choose $\Phi \in V_s$ as initial condition for the spinor and the spinor-DeTurck flow. As above we denote these flows by $\Phi_t$ and $\tilde{\Phi}_t$ respectively and by $f_t$ we mean the associated mapping flow.
We will now estimate the integral of the $H^{-3}$ norm of $\tilde{Q}(\tilde{\Phi}_t)$.
Recall that we have 
$$\int_{T_1}^{T_2} \|\dot{\Phi}_t\|_{L^2} dt \leq C e^{-\alpha T_1}$$
from theorem \ref{L2cv}. For $\dot{f_t}$ we get the estimate
$$\int_{T_1}^{T_2} \|\dot{f}_t\|_{H^{-2}} dt \leq C \left( \int_0^{T_1} \|\dot{g}_t\|_{L^2} e^{\lambda(t-T_1)} dt + \int_{T_1}^{T_2} \|\dot{g}_t\|_{L^2} dt + e^{-\lambda T_1} \right).$$
The second term can be bounded by $C e^{-\alpha T_1}$ by the previous estimate, since $\|\dot{g}_t\|_{L^2} \leq \|\dot{\Phi}_t\|_{L^2}$. The first term we decompose into
$$\int_0^{T_1/2} \|\dot{g}_t\|_{L^2} e^{\lambda(t-T_1)} dt < C e^{-\lambda T_1/2}$$
and
$$\int_{T_1/2}^{T_1} \|\dot{g}_t\|_{L^2} e^{\lambda(t-T_1)} dt  < C e^{-\alpha T_1/2}$$
again using the estimate for $\|\dot{g}_t\|$.
Thus
$$\int_{T_1}^{T_2} \|\dot{f}_t\|_{H^{-2}} dt < C e^{-\mu T}$$
for some $C>0, \mu >0$. We will use the same constants in the estimate of $\dot{g}_t$.
Putting these estimates together we obtain
\begin{eqnarray*}
\int_{T_1}^{T_2} \|\tilde{Q}(\tilde{\Phi}_t)\|_{H^{-3}} dt & \leq & C \int_{T_1}^{T_2} \|\dot{f}_t\|_{H^{-2}} + \|\dot{\Phi}_t\|_{H^{-3}} dt\\
& \leq & C e^{-\mu T_1}
\end{eqnarray*}
Since $\tilde{Q}$ is a continuous map from $H^k$ to $H^{k-2}$, because $H^k$ embeds into $C^3$, and $\tilde{\Phi}_t$ is in a bounded $H^k$ neighborhood, we obtain that $\|\tilde{Q}(\tilde{\Phi}_t)\|_{H^{-3}} \leq \tilde{C}$.
Hence we may estimate
$$\int_{T_1}^{T_2} \|\tilde{Q}(\tilde{\Phi}_t)\|_{H^{-3}}^2 dt \leq \tilde{C} \int_{T_1}^{T_2} \|\tilde{Q}(\tilde{\Phi}_t)\|_{H^{-3}} dt \leq C \tilde{C} e^{-\mu T_1}.$$

Since $\tilde{Q}_t = \tilde{Q}(\tilde{\Phi}_t)$ fulfills the linear strongly parabolic equation
$$\partial_t \tilde{Q}_t = D\tilde{Q}(\tilde{\Phi}_t) \tilde{Q}_t,$$
we may now apply the parabolic estimate \ref{PE}  to obtain
$$\|\tilde{Q}(\tilde{\Phi}_{T + \delta})\|_{H^k} \leq C e^{-\mu T} = \tilde{C} e^{-\mu (T+\delta)}.$$
(Since $\tilde{\Phi}_t$ remains in a bounded neighborhood of $\bar{\Phi}$, the parabolic inequality for $D\tilde{Q}(\tilde{\Phi}_t)$ can be chosen independent of $\tilde{\Phi}_t$.
In particular $\delta$ can be chosen independently of $T$ and $\Phi$, hence the estimate gets worse by a constant factor $e^{\mu \delta}$.)

The argument for the estimate (\ref{GEb}) is identical and for the estimate (\ref{GEc}) the argument runs in parallel until we apply the gradient estimate. Then we get the following estimate:
$$\int_{T_1}^{T_2} \|\dot{\Phi}_t\|_{L^2} dt \leq \frac{C}{1 + T^{\beta}}.$$
Similarly as above, we can estimate
$$\int_{T_1}^{T_2} \|\dot{f}_t\|_{H^{-2}} dt \leq \frac{C}{1 + T^{\beta}}.$$
Thus
$$\int_{T_1}^{T_2} \|\tilde{\mathring{Q}}(\tilde{\Phi}_t)\|_{H^{-3}} dt \leq \frac{C}{1 + T^{\beta}}$$
and hence
$$\|\tilde{Q}(\tilde{\Phi}_t)\|_{H^k} \leq \frac{C}{1 + T^{\beta}}$$
as claimed.
\end{proof}

\begin{proof}[Proof of theorem \ref{stabP}]
In the following $B_\rho$ denotes the ball of radius $\rho$ around $\bar{\Phi}$ with respect to the $H^k$ norm, and in this proof ``flow'' always refers to the gauged spinor flow. 
Using lemmas \ref{UniformExistence} and \ref{GradientDecay}, choose $0 < \gamma < \delta < \epsilon$ and $T$, such that
\begin{enumerate}[label=(\roman*)]
\item The estimate from lemma \ref{GradientDecay} holds on $B_{\epsilon}$.
\item For any $\Phi \in B_{\delta}$ the flow exists until time $1$ and stays in $B_{\epsilon}$
\item $\int_T^{\infty} C e^{-\alpha t} dt < \frac{\delta}{3}$, where $C$ and $\alpha$ as in lemma \ref{GradientDecay}
\item For any $\Phi \in B_{\gamma}$ the flow exists until time $T$ and remains in $B_{\delta/3}$.
\end{enumerate}
Now let $\Phi \in B_{\gamma}$.
Then denote by $\Phi_t$ the flow
$$\partial_t \Phi_t = \tilde{Q}(\Phi_t), \Phi_0 = \Phi.$$
Denote by $\hat{T} \in (0, \infty]$ the maximal time, such that the flow with initial condition $\Phi$ exists in $B_{\delta}$. The condition on $B_{\delta}$ ensures that $\Phi_{\hat{T}}$ exists and $\|\Phi_{\hat{T}} - \bar{\Phi}\|_{H^k} = \delta$.
On the other hand,
\begin{eqnarray*}
\|\bar{\Phi} - \Phi_{\hat{T}}\|_{H^k} & \leq & \|\bar{\Phi} - \Phi_T\|_{H^k} + \|\Phi_T - \Phi_{\hat{T}}\|_{H^k} \\
& \leq & \frac{\delta}{3} + \int_T^{\hat{T}} \|\tilde{Q}(\Phi_t)\|_{H^k} dt \\
& \leq & \frac{\delta}{3} + \int_T^{\hat{T}} C e^{-\alpha t} dt \\
& \leq & \frac{2}{3} \delta
\end{eqnarray*}
This is a contradiction and we conclude $\hat{T} = \infty$.
Additionally, 
$$\int_T^{\infty} \|\tilde{Q}(\Phi_t)\|_{H^k} dt \leq \frac{\delta}{3},$$
and we conclude that the limit
$$\Phi_{\infty} = \Phi_T + \int_T^{\infty} \tilde{Q}(\Phi_t) dt$$
exists in $H^k$ and
$$\|\Phi_{\infty} - \Phi_t\|_{H^k} \leq \int_t^{\infty} \|\tilde{Q}(\Phi_t)\|_{H^k} dt \leq C e^{-\alpha t}.$$
Since 
$$\lim_{t\to \infty} \mc{E}(\Phi_t) = 0,$$
$\Phi_{\infty}$ is a critical point.
We have shown that the gauged spinor flow converges for $\Phi \in B_{\gamma}$ to a critical point in $B_{\delta}$.
Given that the mapping flow is a strongly parabolic equation, the velocity along the flow solves a linear strongly parabolic equation and we can apply the parabolic regularity estimate and the mapping flow estimate to obtain that the mapping flow converges exponentially in any $H^k$ norm. Since the spinor flow is given by $(F_t^{-1})^* \Phi_t$, the spinor flow also converges exponentially.
\end{proof}
\begin{proof}[Proof of theorem \ref{stabVC}]
When the critical set is a manifold, the proof is entirely analogous to the previous proof. If the critical set is not a manifold, we have the weaker estimate
$$\|\tilde{Q}(\Phi_t)\|_{H^k} \leq \frac{C}{1 + T^{\gamma}}.$$
The exponent $\gamma$ can be computed from $\theta$ in the Łojasiewicz inequality as $\gamma = \frac{\theta - 1}{2 - \theta}$. Hence if $\theta > 3/2$, $\gamma > 1$. In that case we find
$$\int_T^{\infty} \frac{C}{1 + t^{\gamma}} dt \leq C\frac{1}{T^{\gamma-1}} \xrightarrow{T \to \infty} 0$$
and we can show existence and convergence of the flow as in the previous proof.
We define
$$\Phi_{\infty} = \Phi_T + \int_T^{\infty} \tilde{Q}(\Phi_t) dt$$
and using that
$$|\mc{E}(\Phi_t) - \mc{E}(\bar{\Phi})| \leq \frac{C}{1+ T^{\beta}}$$
we obtain
$$\mc{E}(\Phi_{\infty}) = \lim_{t\to \infty} \mc{E}(\Phi_t) = \mc{E}(\bar{\Phi})$$
and hence $\Phi_{\infty}$ is also a local minimum, and in particular a critical point of $\mc{E}|_{\mc{N}^1}$.
The speed of convergence is then given by $\frac{1}{T^{\gamma-1}}$.
\end{proof}
\bibliographystyle{plain}
\bibliography{cit}

\end{document}